\documentclass[11pt]{article}
\usepackage{fullpage}
\usepackage{fancyhdr}
\usepackage{amsmath,amssymb,amsthm}
\usepackage{amsmath,amscd}
\usepackage{mathtools}
\usepackage{mathrsfs}
\usepackage{tikz-cd}
\usepackage{authblk}
\usepackage{enumitem}
\usepackage{ stmaryrd }

\usepackage[letterpaper, headsep=1.5cm, margin=1.2in]{geometry}
\usepackage{baskervald}

\usepackage{hyperref}

\newcommand{\Z}{\mathbb{Z}}
\newcommand{\ord}{\text{ord}}

\newcommand{\R}{\mathbb{R}}
\newcommand{\Q}{\mathbb{Q}}

\AtBeginDocument{}

\newtheorem{theorem}{Theorem}[section]
\newtheorem{lemma}[theorem]{Lemma}

\newtheorem{conjecture}[theorem]{Conjecture}
\newtheorem{corollary}[theorem]{Corollary}

\theoremstyle{definition}

\newtheorem{example}[theorem]{Example}
\newtheorem{remark}[theorem]{Remark}
\newtheorem{definition}[theorem]{Definition}
\newtheorem{question}[theorem]{Question}

\def\frak{\relaxnext@\ifmmode\let\next\frak@\else
	\def\next{\Err@{Use \string\frak\space only in math mode}}\fi\next}
\def\goth{\relaxnext@\ifmmode\let\next\frak@\else
	\def\next{\Err@{Use \string\goth\space only in math mode}}\fi\next}
\def\frak@#1{{\frak@@{#1}}}
\def\frak@@#1{\noaccents@\fam\euffam#1}
\font\tengoth=eufm10
\newfam\gothfam \def\goth{\fam\gothfam\tengoth} \textfont\gothfam=\tengoth

\title{$p$-adic estimates of exponential sums on curves}
\author{Joe Kramer-Miller}

\date{}

\begin{document}

	\title{Some unlikely intersections between the Torelli locus and Newton strata
	in $\mathcal{A}_g$}
	\author{Joe Kramer-Miller}

	\date{}

		\maketitle
		
			\begin{abstract}
				Let $p$ be an odd prime. What are the possible Newton
				polygons for a curve in characteristic $p$? Equivalently,
				which Newton strata intersect the Torelli locus in
				$\mathcal{A}_g$? In this note, we study the Newton 
				polygons of certain curves with $\Z/p\Z$-actions. Many of these 
				curves exhibit unlikely intersections between the Torelli locus
				and the Newton stratification in $\mathcal{A}_g$.
				Here is one example of particular interest: fix a genus $g$.
				We show that for any $k$ with $\frac{2g}{3}-\frac{2p(p-1)}{3}\geq 2k(p-1)$,
				there exists a curve of genus $g$ whose Newton polygon has slopes
				$\{0,1\}^{g-k(p-1)} \sqcup \{\frac{1}{2}\}^{2k(p-1)}$. 
				This provides evidence for Oort's conjecture that the amalgamation of
				the Newton polygons of two curves is again the Newton polygon of a curve.
				We also construct families
				of curves $\{C_g\}_{g \geq 1}$, where $C_g$ is a curve of genus $g$,
				whose Newton polygons have interesting asymptotic properties. 
				For example, we construct a family of curves whose Newton polygons are asymptotically
				bounded below by the graph $y=\frac{x^2}{4g}$. The proof uses
				a Newton-over-Hodge result
				for $\Z/p\Z$-covers of curves due to the author, in addition
				to recent work of Booher-Pries on the realization of this Hodge
				bound. 
			\end{abstract}
		
		\section{Introduction}
		Let $p$ be an odd prime. By a curve, we will always mean a smooth proper
		irreducible curve. What are the possible Newton
		polygons for a curve of $g$ curve in characteristic $p$? 
		In general, it seems difficult to answer this question
		for all $g$. However, the following conjecture of Oort
		offers some guidance on what to expect.
		\begin{conjecture} \label{Conjecture: oort}
			(Oort, see \cite[Conjecture 8.5.7]{Oort-problems_automorphisms_of_curves}) Let 
			$C$ (resp. $C'$) be a curve of genus $g$ (resp. $g'$)
			with Newton polygon $P$ (resp. $P'$). 
			Then there exists a curve $C''$ of genus $g+g'$ whose Newton
			polygon $P''$ is the amalgamation of $P$ and $P'$ (i.e. the slopes
			of $P''$ are the disjoint union of the slopes of $P$ and $P'$).
		\end{conjecture}
		\noindent This conjecture implies, for example, that there exists irreducible supersingular curves
		of every genus. This currently only know in characteristic $2$,
		due to a theorem of van der Geer and van der Vlugt (see \cite{vandergeer-derVlugt-supersingular}).
		Another approach to studying Newton polygons of curves is to 
		ask if there are curves $C_g$ of every genus $g \in \Z_{\geq 0}$
		whose Newton polygons approach some limit asymptotically. 
		This motivates the following questions.
		\begin{question} \label{question 1}
			Let $m_1, \dots, m_{2r} \in \Q \cap [0,1]$ such that $m_k=1-m_{2r-k}$. Does there exist
			a family of curves $\{C_g\}_{g \in \Z_{\geq 0}}$, where $C_g$ has
			genus $g$, such that the Newton polygon of $C_g$
			consists only of the slopes $m_1,\dots,m_{2r}$ and each $m_i$ occurs with
			multiplicity close to $\frac{2g}{2r}$?
		\end{question}
		\begin{question} \label{question 2}
			Let $P$ be the graph of a continuous function $f:[0,2] \to [0,1]$. Does there
			exist a family of curves $\{C_g\}_{g \in \Z_{\geq 0}}$, where $C_g$ has
			genus $g$, such that the Newton polygon of $C_g$, scaled
			by a factor of $\frac{1}{g}$, approaches or lies above $P$ as $g\to \infty$?
		\end{question}
		\noindent In this article we study Conjecture \ref{Conjecture: oort}, Question \ref{question 1}, and Question \ref{question 2}
		by considering $\Z/p\Z$-covers of curves. 
		
		\subsection{Some previous results}
		There are two main approaches when trying to find curves with certain Newton polygons.
		The first is to consider $\Z/n\Z$-covers $f:X \to \mathbb{P}^1$. When $n$ is a power
		of $p$ and $f$ is ramified at a single point, the Newton polygon of $X$ has been studied
		extensively by Robba, Zhu, Blache-Ferard, and Liu-Wei, using $p$-adic methods
		pioneered by Dwork (see \cite{Robba-lower_bounds_NP}, \cite{Blache-Ferard-Newton_stratum}, \cite{Liu-Wei-witt-coverings}, and \cite{Zhu-padic_variation_of_Lfunctions}). 
		Also, the work of van der Geer and van der Vlugt (see \cite{vandergeer-derVlugt-supersingular}) studies $\Z/p\Z$-covers with additional structure
		when $p=2$ by analyzing their Jacobians. 
		When $n$ is coprime
		to $p$, the number
		of slope zero segments (i.e. the $p$-rank) was determined in many cases by Bouw (see \cite{Bouw}). Determining the higher slopes appears to be more difficult
		and is the subject of work by Li-Mantovan-Pries-Tang (see \cite{Pries-Li-Wantovan_Tang-Newton_strat}).
		The second approach is to use clutching morphisms between moduli spaces
		of curves. This technique was used by Achter-Pries in \cite{Achter-Pries} to show
		that for $g\geq 4$, there exists a genus $g$ curve whose Newton polygon
		has slopes $\{0,1\}^{g-4}\sqcup \{\frac{1}{4},\frac{3}{4}\}^{4}$. 
		More recently, these two techniques were combined in work
		of Li-Mantovan-Pries-Tang 
		by studying clutching morphisms for tame covers of $\mathbb{P}^1$. 
		Their work constructs many interesting families of curves whose Newton polygons
		are far from ordinary and follow certain patterns. However,
		most of these families do not include curves of every genus. Instead
		these families include curves whose genera satisfy certain congruence conditions. 
		We describe some of these families in Remark \ref{remark about pries-li-mantovan-tang} and Examples \ref{example: main theorem 1}-\ref{example: main theorem 2}.

		\subsection{$\Z/p\Z$-covers with many branch points}
		In \S \ref{section: letting branches go to infinity} we study $\Z/p\Z$-covers of curves with many branch points of fixed Swan conductor.
		Let us state a specific case of our general result (see Theorem \ref{Theorem--many branch points and some are simple} for the general statement).
		\begin{theorem} \label{theorem: specific result double poles}
			For any $g \geq 0$ and $k$ with $\frac{2g}{3}-\frac{2p(p-1)}{3} \geq 2k(p-1)$,
			there exists a curve $C_{g,k}$ of genus $g$ whose Newton polygon has slopes
			$\{0,1\}^{g-k(p-1)} \sqcup \{\frac{1}{2}\}^{2k(p-1)}$. 
		\end{theorem}
		\noindent From Theorem \ref{theorem: specific result double poles} we may deduce Conjecture
		\ref{Conjecture: oort} for many interesting examples. Indeed, if $C=C_{g,k}$ and $C'=C_{g',k'}$, then the Newton polygon of
		$C''=C_{g+g', k+k'}$ is the amalgamation of the Newton polygons of $C$ and $C'$. When $k$
		is not small relative to $g$, we see that $C_{g,k}$ demonstrates
		an unlikely intersection between a Newton stratum and the Torelli locus in
		the Siegel modular variety (see Corollary \ref{corollary: unlikely intersections}).  Furthermore, we may use
		Theorem \ref{theorem: specific result double poles} to study Question \ref{question 1}.
		By letting $k$ be as large as possible,
		we see that there are curves of every genus consisting of slopes $0,\frac{1}{2},1$,
		where each slope occurs with multiplicity approximately $\frac{2g}{3}$. This was previously only known
		under the assumption of $p\equiv 2 \mod 3$ (see \cite[Corollary 9.4]{Pries-Li-Wantovan_Tang-Newton_strat}). 
		
		\subsection{$\Z/p\Z$-covers of curves with large Swan conductor}
		In \S \ref{section: letting swan conductor go to infinity} we study curves 
		with a small number of branch points and large Swan conductors. 
		We construct a family of curves $\{C_g\}_{g \in \Z_{\geq 0}}$, where $C_g$ has genus $g$,
		such that their Newton polygons asymptotically
		lie above or on the graph $y=\frac{x^2}{4g}$ (see Definition \ref{definition: asymptotic lower bounds of NP} and Theorem \ref{theorem: only one branch point} for a precise statement). To the best of our knowledge,
		it is unknown if there is a family of curves containing a curve of every genus
		whose Newton polygons asymptotically lie strictly above $y=\frac{x^2}{4g}$. This would
		certainly follow from Conjecture \ref{Conjecture: oort}. However,
		all of the families constructed in \S \ref{section: letting branches go to infinity}
		and the families constructed in \cite{Pries-Li-Wantovan_Tang-Newton_strat} lie well below
		$y=\frac{x^2}{4g}$ when the genus is large.

		\subsection{Method of proof}
		To prove Theorem \ref{theorem: specific result double poles}
		and Theorem \ref{theorem: only one branch point}, we
		use a Newton-over-Hodge result due to the author
		(see \cite{kramermiller-padic} or Theorem \ref{Theorem: kramer-miller lower bound}).
		This theorem gives a lower bound for the Newton polygon of a $\Z/p\Z$-cover
		$C \to X$
		in terms of local Swan conductors. 
		By considering covers of an arbitrary curve $X$ instead of
		only $\mathbb{P}^1$, we are able to obtain curves of any genus (e.g. if $X$ has genus $i$, the genus of $C$ is of the form $i+k\frac{p-1}{2}$ by Riemann-Hurwitz). 
		This is a clear advantage over earlier techniques.
		For the general statement of Theorem \ref{Theorem--many branch points and some are simple}, we also need to use 
		recent work of Booher-Pries (see \cite{Pries-Booher}). This work shows the lower bound in \cite{kramermiller-padic} is realized
		when certain congruence conditions between $p$ and the Swan conductors hold.

		\subsection{Acknowledgments}
		We would like to thank Jeremy Booher, Raju Krishnamoorthy, Rachel Pries, and
		Vlad Matei for some helpful discussions. 
		
		\section{Newton polygons and unlikely intersections in $\mathcal{A}_g$}
		\subsection{Conventions on Newton polygons}
		
		Let $\alpha \in \mathbb{Z}_{\geq 0} \cup \infty$ and let
		$f:[0,\alpha] \to \mathbb{R}$ be a continuous convex function. We let
		$P=P(f,\alpha)$ refer to the graph of $f$ in the $xy$-plane. We refer to
		the points $(0,f(0))$ and $(\alpha,f(\alpha))$ as the endpoints of $P$. 
			We say that $P$
			is a \emph{symmetric} if 
			$f(\alpha-x)=f(\alpha)-f(x)$ for all $x \in [0,\alpha]$. 
			The following two types of graphs are of particular interest.
		\begin{definition} \label{definition: basic graph}
			A graph $P=P(f,2)$ is called \emph{basic} if $P$
			is symmetric and the endpoints of $P$ are $(0,0)$ and $(2,1)$. 
		\end{definition}
		\begin{definition} We say $P(f,2g)$ is a \emph{Newton polygon
			of height $2g$}
		if $f$ satisfies the following:
		\begin{enumerate}
			\item $f$ is symmetric and the endpoints are $(0,0)$ and $(2g,g)$.
			\item For any integer $i\in (0,2g]$,
			the function $f(x)$ is linear on the domain $x\in [i-1,i]$
			with slope $m_i\geq 0$.
		\end{enumerate}
			We will refer to the multiset $\{m_i\}_{1\leq i \leq 2g}$ as the \emph{slope-set} of $P$
			and its elements as the slopes of $P$. Note that the
			slope-set completely determines the Newton polygon. 
		\end{definition}

		\noindent For $i=1,2$, consider graphs $P_i=P(f_i,\alpha)$. 
		We write $P_1 \succeq P_2$
		if $f_1(x) \geq f_2(x)$ for all $0 \leq x \leq 
		\alpha$. 
		When $P_1 \succeq P_2$, we say that $P_1$ lies above $P_2$.
		If $P_2$ is a Newton polygon with slope-set $N$, we will occasionally write 
		$P_1\succeq N$ instead of $P_1 \succeq P_2$. Finally, for any
		$c\geq 0$, we define the scaled graph $cP_i=\{(cx,cy)|(x,y)\in P_i\}$.

		\subsection{The Newton stratification of $\mathcal{A}_g$}
		Let $\mathcal{A}_g$ denote the moduli space of principally polarized 
		abelian varieties of dimension $g$ and let $\mathcal{X} \to \mathcal{A}_g$ be the universal
		abelian scheme. For each closed $x \in \mathcal{A}_g$, we obtain
		an Abelian variety $\mathcal{X}_x$. Let $NP_x$ denote the
		Newton polygon of height $2g$
		associated to $\mathcal{X}_x$ (see \cite[(1.2)]{Oort-Newton_polygon_strata_dimensions}).  We remark that the vertices of $NP_x$ have integer coordinates. This greatly restricts the possibilities for $NP_x$. 
		
		\begin{definition}
			Let $P$ be a Newton polygon of height $2g$.
			Let $W_P \subset \mathcal{A}_g$ (resp. $W_P^0$) denote the locus of principally polarized 
			abelian varieties $\mathcal{X}_x$ with  $NP_x \succeq P$ (resp. $NP_x=P$). 
			Note that $W_P$ is closed in $\mathcal{A}_g$ and $W_P^0$ is open in $W_P$
			(see \cite[\S 4]{Oort-Newton_polygon_strata_dimensions}).
		\end{definition}
		
		\begin{theorem} \label{Theorem: Oort's lattice count theorem}
			The codimension of $W_P$ in $\mathcal{A}_g$ is at least 
			$\#\Omega(P)$, where 
			\begin{align*}
				\Omega(P) &= \{(x,y) \in \Z_{\geq 0} \times \Z_{\geq 0} ~|~0\leq x \leq g \text{ and }(x,y) \text{ lies strictly below } P\}.
			\end{align*}
			Furthermore, if $P$ has integer vertices, the codimension of $W_P$
			in $\mathcal{A}_g$ is exactly $\#\Omega(P)$.
		\end{theorem}
		\begin{proof}
			This follows from \cite[Theorem 4.1]{Oort-Newton_polygon_strata_dimensions},
			by noting that $W_P = \cup_{P'} W_{P'}$, where the union
			runs over all Newton polygons of height $2g$ with integer vertices lying above $P$.
		\end{proof}
		
		\subsection{The Torelli locus}
		The Torelli map $\iota: \mathcal{M}_g \to \mathcal{A}_g$ sends a curve
		$C$ of genus $g$ to its Jacobian. 
		Let $\mathcal{T}_g$ denote
		the image of the Torelli map. It is a closed subscheme of dimension $3g-3$. 
		The Newton polygon $NP(C)$ of $C$ is defined to be the Newton polygon
		of the corresponding point in $\mathcal{A}_g$. If $C$ is defined
		over a finite field $\mathbb{F}_q$, then $NP(C)$ is equal to the $q$-adic Newton polygon
		of the numerator of its zeta function. We define the \emph{scaled Newton polygon} to be
		\begin{align*}
			sNP(C)&=\frac{1}{g} NP(C).
		\end{align*}
		Note that $sNP(C)$ is a basic graph (see Definition \ref{definition: basic graph}).
		We are interested in the following questions.
		\begin{question}  \label{question: individual curves}
			How does the Torelli locus interact with the Newton stratification? More specifically,
			let $P$ be a Newton polygon of height $2g$. Does there exist a curve
			$C$ of genus $g$ with $NP(C)=P$ (resp. $NP(C) \succeq P$). Equivalently, is 
			$W_P^0 \cap \mathcal{T}_g$ (resp. $W_P\cap \mathcal{T}_g$) nonempty?
		\end{question}
		\noindent This question appears to be very difficult in general. Instead, one may ask
		for asymptotic behavior as $g$ gets large. One way to do this is to study
		the behavior of $sNP(C)$ as $C$ varies over a collection of curves. This prompts the following definitions.
		
		\begin{definition}
			A \emph{family of curves} is a collection $\mathcal{C}=\{C_g\}_{g \in S}$ where $S \subset \Z_{\geq 0}$ and $C_g$ is a curve of genus $g$.
			We say that $\mathcal{C}$ is \emph{full} if $S=\Z_{\geq 0}$ and we say that $\mathcal{C}$
			is \emph{arithmetic} if $S$ is the union of finitely many arithmetic 
			progressions.
		\end{definition}
		
		\begin{definition}\label{definition: asymptotic lower bounds of NP}
			Let $\mathcal{C}=\{C_g\}_{g \in S}$ be a
			family of curves and let $P=P(f,2)$ be a basic graph. For each $g \in S$, let $f_g$ be the piecewise linear function
			such that $sNP(C_g)=P(f_g,2)$. We write $\mathbf{sNP}(\mathcal{C}) \succeq P$ if
			\begin{align*}
				\liminf_{g \in S} \min_{x \in [0,2]} \{f_g(x)-f(x)\} &\geq 0.
			\end{align*}
		\end{definition}
	
		\begin{question} \label{questions: full families of curves}
			Let $P$ be a basic graph. Does there exist a family of curves $\mathcal{C}$ such that
			$\mathbf{sNP}(\mathcal{C})\succeq P$?
			Can we take this family to be arithmetic or full?
		\end{question}
		
		\noindent Another natural question is to ask for families of curves where certain slopes
		occur with some specified frequency. For example, we may ask for a family of curves
		$\{C_g\}_{g \in S}$ where $NP(C_g)$ only has slopes $0,\frac{1}{2},1$, and each slope occurs
		with approximately equal frequency. This prompts the following definition. 
		\begin{definition} \label{definition: specified slope families}
			 Let $\mathcal{C}=\{C_g\}_{g \in S}$ be a
			 family of curves. 
			 Let $\mathcal{N}=\{m_i\}_{i=1}^{2r}$ be the slope-set of
			 a Newton polygon with $m_i \in [0,1]\cap \mathbb{Q}$.
			 We write $\mathbf{NP}(\mathcal{C}) \sim \mathcal{N}$
			 if there exists $\epsilon>0$ such that
			 \begin{align*}
				 NP(C_g) = \bigsqcup_{i=1}^{2r} \big \{m_i\big \}^{\frac{2g}{2r} + e_i(g)},
			 \end{align*} 
			 with $|e_i(g)|<\epsilon$. Informally, this means that $NP(C_g)$ has slopes $m_1,\dots,m_{2r}$, each occurring with approximately
			 the same frequency. 
		\end{definition}
	
		\begin{question}
			Let $\mathcal{N}=\{m_i\}_{i=1}^{2r}$ be as in Definition \ref{definition: specified slope families}. Is there a full or arithmetic family $\mathcal{C}$ with $\mathbf{NP}(\mathcal{C}) \sim \mathcal{N}$?
		\end{question}

		\begin{remark}
			If one does not require the family of curves to be arithmetic or full, it is 
			much easier to find families with interesting asymptotic properties.
			For example, in \cite[Corollary 2.6]{Pries-Newton_survey} the authors construct
			an infinite family of supersingular curves. However, this
			family is much too sparse to be arithmetic.  
		\end{remark}
		
		\begin{remark} \label{remark about pries-li-mantovan-tang}
			The work of \cite{Pries-Li-Wantovan_Tang-Newton_strat} proves the existence of many families 
			$\mathcal{C}$ satisfying $\mathbf{NP}(\mathcal{C}) \sim \mathcal{N}$ where
			$\mathcal{N}$ is some interesting slope-set.
			Most of these families are arithmetic, although a few special cases are full. 
			Here is one particularly interesting example: assume that $p \equiv 4 \mod 5$.
			They prove the existence of a family $\mathcal{C}=\{C_g\}_{g \in S}$
			where $S=\{10n-4\}_{n\geq 1}$ and
			\begin{align*}
				\mathbf{NP}(\mathcal{C}) &\sim \{0\}^3 \sqcup \{\frac{1}{2}\}^{4} \sqcup \{1\}^3.
			\end{align*}
			The key technical aspect of their work is an analysis of clutching morphisms
			for moduli of tame cyclic covers of $\mathbb{P}^1$. Using this analysis,
			they give an inductive process to construct arithmetic families of curves
			with prescribed Newton polygon.
			Combining this with their previous results on special subvarieties of Shimura varieties (see \cite{Li-Mantovan-Pries-Tang-special_families})
			gives many interesting examples.
		\end{remark}
		
		\subsection{Unlikely intersections on $\mathcal{A}_g$}
		Let $X$ be a variety of dimension $d$. Let $V_1$ and $V_2$ be
		subvarieties of $X$ with codimensions $c_1$ and $c_2$. If $c_1+c_2\leq d$,
		then we expect the intersection $V_1 \cap V_2$ to be nonempty and have dimension $d-c_1-c_2 \geq 0$. However, if $c_1+c_2>d$, then $V_1 \cap V_2$ will typically be empty. 
		We say that $V_1$ and $V_2$ have an \emph{unlikely intersection} 
		if $V_1 \cap V_2$ is nonempty and $c_1+c_2>d$. For example,
		let $\mathcal{A}^{s.s.}_g\subset \mathcal{A}_g$ denote the supersingular locus.
		We know that $\dim(\mathcal{A}_g^{s.s.})=\lfloor \frac{g^2}{4} \rfloor$
		and $\dim(\mathcal{T}_g)=3g-3$. Thus, the existence of a supersingular curve
		of genus $g>3$ implies $\mathcal{A}_g^{s.s.}$ and $\mathcal{T}_g$ have
		an unlikely intersection. 
		More generally, a high genus curve that is sufficiently far from
		being ordinary implies an unlikely intersection between the Torelli
		locus and a Newton stratum. We point the reader to Oort's article 
		in \cite{Oort-problems_automorphisms_of_curves} for more background.

		\begin{definition}
			Let $C$ be a curve of genus $g$ and let $P=NP(C)$. We say that $C$ has
			an \emph{unlikely Newton polygon} if $W_P^0$ and
			$\mathcal{T}_g$ have an unlikely intersection. We say a family of curves $\mathcal{C}=\{C_g\}_{g\in S}$ 
			is \emph{unlikely} if for $g\gg 0$ the curve
			$C_g$ has an unlikely Newton polygon.
		\end{definition}
		
		\begin{lemma} \label{lemma: unlikely family}
			Let $\mathcal{C}=\{C_g\}_{g\in S}$ be a family of curves and let
			$P=P(f,2)$ be a basic graph. If $f(1)>0$
			and $\mathbf{sNP}(\mathcal{C})\succeq P$, then $\mathcal{C}$ is an unlikely family.
		\end{lemma}
	
		\begin{proof}
			We may replace $P$ with a slightly lower basic graph so that
			$sNP(C_g)\succeq P$ for large $g$. By lowering $P$ more we may assume that $P$
			consists of three line segments with slopes $0$, $\frac{1}{2}$, and $1$.
			The codimension of $\mathcal{T}_g$ in $\mathcal{A}_g$ is
			$\frac{g(g+1)}{2}-3g-3$, so it suffices to show that the codimension
			of $W_{gP}$ in $\mathcal{A}_g$ grows quadratically in $g$.
			Define $R_g=\{ (x,y)\in\R\times \R ~|~ 0\leq x \leq g \text{ and } 0 \leq y < gf(\frac{x}{g}) \}$.
			By Theorem \ref{Theorem: Oort's lattice count theorem},
			it suffices to show that $\# ((\Z \times \Z) \cap R_g)$
			grows quadratically in $g$. This follows by observing that the $R_g$ are similar
			triangles whose 
			side lengths grow linearly in $g$. 
		\end{proof}

		\begin{corollary} \label{corollary: unlikely family}
			Let $\mathcal{C}$ be a family of curves and let $\mathcal{N}$
			be a Newton polygon of height $2g$ that is non-ordinary. If $\mathbf{NP}(\mathcal{C}) \sim \mathcal{N}$,
			then $\mathcal{C}$ is an unlikely family. 
		\end{corollary}
		\begin{proof}
			Let $P(f,2)=\frac{1}{g}\mathcal{N}$. Note that $P(f,2)$ is a basic graph. The non-ordinary
			condition implies $f(1)>0$ and the condition $\mathbf{NP}(\mathcal{C}) \sim \mathcal{N}$
			implies $\mathbf{sNP}(\mathcal{C}) \succeq P(f,2)$. The corollary follows from Lemma \ref{lemma: unlikely family}.
		\end{proof}
		
		\section{$\Z/p\Z$-covers of curves}
		\label{section: p-covers}
		Let $X$ be a curve of genus $g$ over a finite field $\mathbb{F}_q$
		and let $r:C\to X$ be a $\Z/p\Z$-cover. Let $\tau_1,\dots,\tau_m\in X$
		be the points where $r$ is ramified. Let $d_i$ be the Swan conductor
		of $r$ at $\tau_i$. We may describe $d_i$ as follows. Let 
		$t_i$ be
		a local parameter at $\tau_i$. Locally 
		the cover $r$ is given by an Artin-Schreier equation $Y^p-Y=g_i$, where $g_i \in \mathbb{F}_q((t_i))$.
		We may assume that $g_i$ has a pole whose order is coprime to $p$.
		The order of this pole is equal to the Swan conductor. That is, 
		$g_i=\sum\limits_{n\geq-d_i} 
		a_nt_i^n$ and $a_{-d_i}\neq 0$. Now, let $f$ be a rational
		function on $X$ and assume that $C$ is given by the equation
		$Y^p-Y=f$. Then $-\ord_{\tau_i}(f) \geq d_i$ with equality if
		and only if $\gcd(\ord_{\tau_i}(f),p)=1$.

		\begin{theorem} \label{Theorem: kramer-miller lower bound}
			We have
			\[  NP(C) \succeq \Bigg\{0,1\Bigg \}^{pg + m(p-1)}\bigsqcup
			\Bigg\{ \frac{1}{d_1}, \dots, \frac{d_1-1}{d_1}, \dots, 
			\frac{1}{d_{m}}, \dots, \frac{d_{m}-1}{d_{m}} \Bigg\}^{p-1}.\]
		\end{theorem}
		\begin{proof}
			This is an earlier theorem of the author. See \cite[Corollary 1.3]{kramermiller-padic}.
			The proof uses the Monsky trace formula and some delicate $p$-adic analysis.
		\end{proof}
	
		\begin{corollary} \label{corollary: bound is optimal for simple poles}
			Assume $X$ is ordinary and each Swan conductor is equal to $2$. Then
			\begin{align} \label{equation: lower bound for double poles}
				NP(C)&= \Bigg\{0,1\Bigg \}^{pg + m(p-1)}\bigsqcup
				\Bigg\{ \frac{1}{2} \Bigg\}^{m(p-1)}.
			\end{align} 
		\end{corollary}
	
		\begin{proof}
			From Theorem \ref{Theorem: kramer-miller lower bound} we know 
			$NP(C)\succeq \{0,1 \}^{pg + m(p-1)}\sqcup
			\{ \frac{1}{2} \}^{m(p-1)}$.
			By the Deuring-Shafarevich formula (see e.g. \cite{Crew-p-covers}),
			we see that $\{0,1 \}^{pg + m(p-1)}\sqcup
			\{ \frac{1}{2} \}^{m(p-1)}$ has the correct number of
			slope zero segments and slope one segments. As the remaining
			slopes are $1/2$, the two Newton polygons must be equal. 
		\end{proof}
		\noindent In general, the bound in Theorem \ref{Theorem: kramer-miller lower bound}
		will not be attained.
		However, if $p\equiv 1 \mod d_i$ for each $i$, recent work of Booher-Pries
		shows this bound is optimal.
		\begin{theorem}  \label{Theorem:booher-pries}
			(Booher-Pries) Assume $X$ is ordinary and let $d_1,\dots,d_m \in \Z_{\geq 1}$
			such that $p\equiv 1 \mod d_i$ for each $i$. There exists
			a $\Z/p\Z$-cover of $X$, which is ramified at the points $\tau_1,\dots,\tau_m$ with
			Swan conductor $d_i$, such that
				\[ NP(C)= \Bigg\{0,1\Bigg \}^{pg + m(p-1)}\bigsqcup
				\Bigg\{ \frac{1}{d_1}, \dots, \frac{d_1-1}{d_1}, \dots, 
				\frac{1}{d_{m}}, \dots, \frac{d_{m}-1}{d_{m}} \Bigg\}^{p-1}.\]
		\end{theorem}
		\begin{proof}
			See \cite[Corollary 4.3]{Pries-Booher}. The main idea is as follows: work of Blache-Ferard computes the
			Newton polygon of
			a generic $\Z/p\Z$-cover of $\mathbb{P}^1$ ramified only at $\infty$ when the Swan conductor is less than $3p$. 
			Booher and Pries use this to construct a $\Z/p\Z$-cover of singular curves $C_0 \to X_0$
			and calculate the Newton polygon of $C_0$. Using a formal patching argument,
			they show that $C_0 \to X_0$ deforms to a family of $\Z/p\Z$-covers $\mathcal{C}_0 \to \mathcal{X}_0$ that generically gives a cover of $X$. By Grothendieck's specialization theorem (see \cite{katz-slope_filtration}), this gives an upper bound for the Newton polygon of a generic cover in this family. This upper bound is precisely the lower bound of Theorem \ref{Theorem: kramer-miller lower bound}
			when $p \equiv 1 \mod d_i$. 
		\end{proof}
	
		\section{Letting the number of branch points tend to infinity}
		\label{section: letting branches go to infinity}
			
			\begin{theorem} \label{Theorem--many branch points and some are simple}
				Let $d\geq 2$ with $p\nmid d$. Set $\delta$ to be $1$ if $d$ is odd and $2$ if $d$ is even. For $g\geq 1$ and $k$ satisfying
				\begin{align} \label{bound on k eq}
					\frac{2g}{d+1} - \frac{2p(p-1)}{d+1} &\geq k\delta(p-1),
				\end{align}
				there exists a curve $C_{g,k}$ of genus $g$ such that
				\begin{align} \label{precise number of middle slopes eq}
					NP(C_{g,k}) &\succeq \Big\{0,1\Big\}^{g-\frac{k\delta(p-1)(d-1)}{2}} 
					\bigsqcup \Big\{\frac{1}{d}, \dots, \frac{d-1}{d}\Big \}^{k\delta(p-1)}.
				\end{align}
				If $p\equiv 1 \mod d$, then we may choose $C_{g,k}$
				so that \eqref{precise number of middle slopes eq} is
				an equality of Newton polygons. 
			\end{theorem}
			
			\begin{proof}
				Write $g=i+m(p-1)$, where $0\leq i < p-1$, and define
				\begin{align*}
					A&= g-ip+(p-1) - \frac{k\delta(p-1)(d+1)}{2}.
				\end{align*}
				By \eqref{bound on k eq} we know $A \geq 0$. Also, we see that $(p-1)|A$.
				Let $j\geq 0$ with $A=j(p-1)$. Choose an ordinary curve $X_i$
				with genus $i$. Then choose an Artin-Schreier cover $C_{g,k} \to X_i$
				ramified at $j+\delta k$ points, such that $j$ points have
				Swan conductor $1$ and $\delta k$ points have Swan conductor 
				$d$. By Riemann-Hurwitz we know $C_{g,k}$ has
				genus $g$. We then apply Theorem \ref{Theorem: kramer-miller lower bound} to obtain the bound
				\eqref{precise number of middle slopes eq}. If $p \equiv 1 \mod d$
				we can use Theorem \ref{Theorem:booher-pries} (or Corollary \ref{corollary: bound is optimal for simple poles} when $d=2$) to make sure the
				Newton polygons in \eqref{precise number of middle slopes eq} are equal. 
			\end{proof}
			
			\begin{corollary} \label{Theorem--many branch points}
				Let $d \geq 2$ with $p\nmid d$. Let $\mathcal{N}$
				be the Newton polygon with slopes $\{0,1\}^u \sqcup \{\frac{1}{d}, \dots, \frac{d-1}{d} \}^v$ where $u\geq v$ and let $\mathcal{P}=\frac{1}{2u+(d-1)v}\mathcal{N}$ be
				the scaled Newton polygon. There exists
				a full family of curves $\mathcal{C}$ such that 
				$\mathbf{sNP}(\mathcal{C}) \succeq  \mathcal{P}$.
				In particular, $\mathcal{C}$ is an unlikely family.
				If $p\equiv 1 \mod d$, we may choose $\mathcal{C}$ so that
				$\mathbf{NP}(\mathcal{C}) \sim \mathcal{N}$.
			\end{corollary}
			 \begin{proof}
			 	The existence of $\mathcal{C}$ follows from Theorem \ref{Theorem--many branch points and some are simple}. From Lemma \ref{lemma: unlikely family} and Corollary \ref{corollary: unlikely family} we know $\mathcal{C}$ is an unlikely family.
			 \end{proof}

			\begin{corollary}
				Let $d\geq 2$ and assume $p\equiv 1 \mod d$.
				Let $g,g'\geq 1$.
				Let $C_{g,k}$ and $C_{g',k'}$ be curves as in Theorem \ref{Theorem--many branch points and some are simple}.
				Then Oort's conjecture (see Conjecture \ref{Conjecture: oort}) holds for
				$C_{g,k}$ and $C_{g'k'}$. That is, there exists a curve $C$ of genus $g+g'$ 
				whose Newton
				polygon is $NP(C_{g,k}) \sqcup NP(C_{g',k'})$.
			\end{corollary}
			\begin{proof}
				Take $C$ to be $C_{g+g',k+k'}$. 
			\end{proof}
		
			\begin{corollary}\label{corollary: unlikely intersections}
				Set $d=2$. Let $g$ and $k$ be as in Theorem \ref{Theorem--many branch points and some are simple}. Define $m=\frac{p-1}{2}$. Let $\mathcal{N}_{g,k}$ denote the Newton polygon with slopes
				$\{0,1\}^{g-k(p-1)} \sqcup \{\frac{1}{2}\}^{2k(p-1)}$.
				If $km(km-1)>3g-g$, then $W_{\mathcal{N}_{g,k}}$ and $\mathcal{T}_g$ have an unlikely
				intersection.
			\end{corollary}
			
			\begin{proof}
				By Theorem \ref{Theorem--many branch points and some are simple} we know $W_{\mathcal{N}_{g,k}}$ and $\mathcal{T}_g$ have a nonempty intersection.
				We will compute
				the codimension of $W_{\mathcal{N}_{g,k}}$ in $\mathcal{A}_g$
				using Theorem \ref{Theorem: Oort's lattice count theorem}.
				Let $T_{g,k}$ be the triangle whose vertices are $(g-k(p-1),0)$,
				$(g,0)$ and $(g, \frac{k(p-1)}{2})$. We compute
				$\#(T_{g,k} \cap \Z \times \Z)=(km)^2$. Also,
				there are $km$ lattice points on the hypotenuse of $T_{g,k}$. 
				The codimension of 
				$W_{\mathcal{N}_{g,k}}$ in $\mathcal{A}_g$ is thus $km(km-1)$.
				The codimension of $\mathcal{T}_g$ in $\mathcal{A}_g$ is $\frac{g(g-1)}{2}-3g-3$,
				which proves the corollary.
			\end{proof}

			\begin{example} \label{example: main theorem 1}
				Consider the case where $d=2$ and let $\mathcal{N}=\{0,\frac{1}{2},1\}$.
				Corollary \ref{Theorem--many branch points} 
				tells us there exists a full family of curves $\mathcal{C}=\{C_g\}$,
				where $NP(C_g)$ only has slopes $0$,$1$, and $\frac{1}{2}$,
				and each occurs about a third of the time (with a constant error term). This was previously only known under the assumption $p\equiv 2 \mod 3$ (see \cite[Corollary 9.4]{Pries-Li-Wantovan_Tang-Newton_strat}).
				More generally, let $\epsilon\leq \frac{1}{3}$ be a rational number.
				There exists a full family $\mathcal{C}=\{C_g\}$, such that the
				Newton polygon of $C_g$ consists only of slopes $0,1$, and $\frac{1}{2}$
				and the multiplicity of $\frac{1}{2}$ is approximately $2g\epsilon$ (with a constant error term).
			\end{example}
			
			\begin{example} \label{example: main theorem 2}
				Consider the case where $d=3$ and $p\equiv 1 \mod 3$. Let
				$\mathcal{N}=\{0,\frac{1}{3},\frac{2}{3}, 1\}$. Corollary \ref{Theorem--many branch points} gives a full family of curves $\mathcal{C}=\{C_g\}$,
				where $NP(C_g)$ only has slopes $0$,$\frac{1}{3}$, $\frac{2}{3}$, and $1$.
				Furthermore, each slope occurs with approximately equal frequency.
				To the best of our knowledge, there were no previous examples of full or arithmetic families of
				curves with this property. There were, however, examples
				of arithmetic families of curves where the slopes $\frac{1}{3},\frac{2}{3}$ occurred
				with smaller frequencies (see \cite[\S 9.2]{Pries-Li-Wantovan_Tang-Newton_strat}). 
			\end{example}

		\section{Letting the ramification break tend to infinity} 
		\label{section: letting swan conductor go to infinity}

		\begin{theorem} \label{theorem: only one branch point}
			There exists a full family $\mathcal{C}=\{C_g\}$ such that
			\begin{align*}
				\mathbf{sNP}(\mathcal{C}) &\succeq P\Big (\frac{x^2}{4},2\Big ).
			\end{align*}
			In particular $\mathcal{C}$ is an unlikely family.
		\end{theorem}
	
		\begin{proof}
			Let $m=\frac{p-1}{2}$ and for each $i=0,\dots, m-1$, choose
			$u_i,v_i$ such that 
			\begin{align*}
				pu_i- (p-1) &=  i + mv_i.
			\end{align*}
			Let $X_i$ be a smooth ordinary curve of genus $u_i$.
			For $g\gg 0$, write $g=i+km$. We define a genus $g$ curve $C_g$
			as follows:
			\begin{enumerate}[label=\Roman*.]
				\item If $d=k-1-v_i$ is relatively prime to $p$, we choose
				a rational function $f_g$ on $X_i$ that has exactly one pole of 
				order $d$ (this is possible by Riemann-Roch). We let $C_g$ be the curve defined by the Artin-Schreier
				equation $y^p-y=f_g$. By the Riemann-Hurwitz formula we know 
				$C_g$ has genus $g$.
				\item If $d=k-1-v_i$ is divisible by $p$, we choose a rational
				function $f_g$ that has exactly two poles: one pole of of order
				$d-2$ and one pole of order $2$ (again, this is possible by Riemann-Roch).
				We let $C_g$ be the curve defined by the Artin-Schreier
				equation $y^p-y=f_g$. By the Riemann-Hurwitz formula we know 
				$C_g$ has genus $g$.
			\end{enumerate}
			The theorem then follows from Theorem \ref{Theorem: kramer-miller lower bound}.
			Indeed, by Theorem \ref{Theorem: kramer-miller lower bound} we know that
			$NP(C_g)$ is very close to the Newton polygon 
			$\{0,\frac{1}{g},\dots,\frac{g-1}{g} \}$. The latter polygon, when 
			scaled by $\frac{1}{g}$,
			approaches the graph $P\Big (\frac{x^2}{4},2\Big )$ as $g$ tends to infinity.
		\end{proof}

		\begin{question}
			Theorem \ref{theorem: only one branch point} allows for the possibility
			that $sNP(C_g)$ lies well above $P(\frac{x^2}{4},2)$.  
			Does there exist a full family of curves $\mathcal{C}=\{C_g\}_{g \in \Z_{\geq 0}}$ 
			such that $sNP(C_g)$ converges uniformly to $P(\frac{x^2}{4},2)$?
		\end{question}
		\begin{question}
			Does there exist a full or arithmetic family of curves $\mathcal{C}$ and a basic graph $P_0$ that lies strictly above $P(\frac{x^2}{4},2)$
			such that $\mathbf{sNP}(\mathcal{C}) \succeq P_0$? 
		\end{question}
		To the best of our knowledge, both questions
		are unknown even for arithmetic families. The arithmetic families constructed in \cite{Pries-Li-Wantovan_Tang-Newton_strat} have scaled Newton
		polygons whose limits are well below $P(\frac{x^2}{4},2)$. Similarly, the bounds in Theorem \ref{Theorem--many branch points} are well below $P(\frac{x^2}{4},2)$.

	\bibliographystyle{plain}
	\bibliography{bibliography.bib}

\end{document}